\documentclass{siamart190516}

\usepackage{amsmath,amsfonts,amssymb}

\newtheorem{remark}{Remark}[section]

\newtheorem{thm}{Theorem}[section]
\newtheorem{cor}{ Corollary}[section]
\newtheorem{prop}{Proposition}[section]
\newtheorem{lem}{Lemma}[section]

\usepackage{graphicx}                          
\usepackage{setspace}  
\usepackage{caption}

\usepackage{listings}
\usepackage{tabularx}
\usepackage{xcolor}
\usepackage{indentfirst}

\usepackage{overpic}
\usepackage{diagbox}

\usepackage{verbatim}

\usepackage{enumerate}

\usepackage[normalem]{ulem}

\usepackage{tikz}

\usepackage[title]{appendix}

\usepackage{bm}

\usepackage[font=footnotesize,labelfont=bf]{caption}

\usepackage{nicefrac}

\newcommand\dd{\mathrm{d}}
\newcommand\pp{\partial}

\newcommand\x{\bm{x}}
\newcommand\uvec{\mathbf{u}}

	\newcommand\be {\begin{equation}}
	\newcommand\ee {\end{equation}}
	\newcommand\bu {{\bf u}}

	\newcommand\dt {{\Delta t}}

\usepackage{mhchem}

\usepackage{multirow}

\begin{document}

\title{A second order accurate, operator splitting scheme for reaction-diffusion systems in an energetic variational formulation}

\author{Chun Liu\thanks{Department of Applied Mathematics, Illinois Institute of Technology, Chicago, IL 60616, USA (Corresponding author: \email{cliu124@iit.edu})} \and Cheng Wang\thanks{Department of Mathematics, University of Massachusetts, North Dartmouth, MA 02747 (\email{cwang1@umassd.edu})}\and Yiwei Wang\thanks{Department of Applied Mathematics, Illinois Institute of Technology, Chicago, IL 60616, USA (\email{ywang487@iit.edu})} }

\maketitle

\begin{abstract}
A second-order accurate in time, positivity-preserving, and unconditionally energy stable operator splitting numerical scheme is proposed and analyzed for the system of reaction-diffusion equations with detailed balance.  The scheme is designed based on an energetic variational formulation, in which the reaction part is reformulated in terms of the reaction trajectory, and both the reaction and diffusion parts dissipate the same free energy. 
At the reaction stage, the reaction trajectory equation is approximated by a second-order Crank-Nicolson type method. The unique solvability, positivity-preserving and energy-stability are established based on a convexity analysis. In the diffusion stage, an exact integrator is applied if the diffusion coefficients are constant, and a Crank-Nicolson type scheme is applied if the diffusion process becomes nonlinear.
In either case, both the positivity-preserving property and energy stability could be theoretical established. 
Moreover, a combination of the numerical algorithms at both stages by the Strang splitting approach leads to a second-order accurate, structure preserving scheme for the original reaction-diffusion system. Numerical experiments are presented, which demonstrate the accuracy of the proposed scheme.
\end{abstract}

\section{Introduction}

In this work, we consider the following type of reaction diffusion systems 
\begin{equation}\label{RD_1}
  \pp_t c_i = \nabla \cdot (D_i(c_i, \x) \nabla c_i) + r_i ({\bm c}), \quad i = 1, \ldots N,
 \end{equation}
where $c_i > 0$ is the concentration of $i$-th species, $D_i(c_i, \x)$ are diffusion coefficients, and $r_i ({\bm c})$ are nonlinear reaction terms for the chemical reaction
\begin{equation}
  \ce{ $\alpha_{1}^{l} X_1 + \alpha_{2}^{l}X_2 + \ldots \alpha_{N}^{l} X_N$ <=> $\beta_{1}^{l} X_1 + \beta_{2}^{l}X_2 + \ldots \beta_{N}^{l} X_N$}, \quad l = 1, \ldots, M.
  \end{equation}
Such a type of reaction-diffusion systems can be found in many mathematical models in chemical engineering, biology, soft matter physics and combustion theory, see \cite{chipot2003variational, hao2020spatial, hawkins2012numerical, julicher1997modeling, kondo2010reaction, liu2018accurate, pearson1993complex, perthame2014hele, prost2015active, wang2003robust,  wang2021two} for examples.

Numerical simulation for the reaction-diffusion system (\ref{RD_1}) turns out to be very challenging, due to the stiffness brought by the reaction term. Moreover, a naive discretization to (\ref{RD_1}) may fail to preserve the positivity and the conservation property in the original system \cite{formaggia2011positivity}. 
 To overcome these difficulties, many numerical methods have been developed to solve reaction kinetics and reaction-diffusion systems \cite{bertolazzi1996positive, formaggia2011positivity, huang2019positivity, zhao2011operator}, including some operator splitting approaches~\cite{carrillo2018splitting, descombes2001convergence, gallouet2019unbalanced, gallouet2017jko, zhao2011operator}.
 
It has been discovered that for certain form of reaction-diffusion systems, in which the reaction part describes the reversible chemical reaction satisfying the law of mass action with detailed balance condition, the whole system admits an energy-dissipation law, which opens a door of developing structure-preserving numerical schemes.
In more details, under certain conditions, which will be specified in the next section, the reaction-diffusion system (\ref{RD_1}) can be reformulated as a combination of two generalized gradient flows (with different patterns) for a single free energy \cite{liero2013gradient, wang2020field}.
Since the reaction and diffusion parts of the original system dissipate the same free energy, it is natural to use an operator splitting approach to develop an energy stable scheme for the whole system. 
Based on this variational structure, a first order accurate operator splitting scheme has been constructed in a recent work~\cite{liu2020structure}, with the variational structure theoretically preserved for the numerical solution. In this approach, since the physical free energy is in the form of logarithmic functions of the concentration ${\bm c}$, a linear function of reaction trajectories ${\bm R}$, the positivity-preserving analysis of the numerical scheme at both stages has been established. Similar to the analysis in a recent article~\cite{chen19b} for the Flory-Huggins Cahn-Hilliard flow, an implicit treatment of the nonlinear singular logarithmic term is crucial to theoretically justify its positivity-preserving property. A more careful analysis reveals that, the convex and the singular natures of the implicit nonlinear parts prevent the numerical solutions approach the singular limiting values, so that the positivity-preserving property is available for the density variables of all the species. A detailed convergence analysis and error estimate have also been reported in a recent work~\cite{LiuC2021c}. However, 
it is a not trivial task to develop a second order accurate operator splitting scheme based on this idea. In fact, most existing works of second order energy stable scheme for gradient flows are multi-step algorithms, based on either modified Crank-Nicolson or BDF2 temporal discretization, and a multi-step approximation to the concave terms is usually needed to ensure both the unique solvability and energy stability. On the other hand, a single step, second order approximation has to be accomplished at each stage in the operator splitting approach, while a theoretical justification of positivity-preserving and energy stability turns out to be very challenging. 

In this article, we propose and analyze a second order accurate operator splitting scheme for the reaction-diffusion system with the detailed balance condition. Following the energetic variational formulation, the splitting scheme solves the reaction trajectory equation of ${\bm R}$ at the reaction stage, and solves the diffusion equation for ${\bm c}$ in the diffusion stage.  To overcome the above-mentioned difficulties, we make use of a numerical profile created by the first order convex splitting algorithm, which is proved to be a second order accurate approximation to the physical quantity at time step $t^{n+1}$, to construct a second order approximation to the mobility part. 
Then an application of modified Crank-Nicolson formula leads to a second order approximation to the mobility function at the intermediate time instant $t^{n+1/2}$. Meanwhile, the physical energy does not contain any concave part in the reaction-diffusion system, so that a single step, modified Crank-Nicolson method leads to a second order accurate algorithm. In addition, an artificial second order Douglas-Dupont-type regularization term \cite{chen19b}, in the form of $\dt \sum_{i=1}^N \sigma_i ( \mu_i (R^{n+1} - R^n) ))$, is added in the chemical potential, to ensure the positivity-preserving property. 
The energy stability is derived by a careful energy estimate, because of the choice in the modified Crank-Nicolson approximation. These techniques lead to a second order accurate, positivity preserving and energy stable algorithm in the reaction stage. 

  In the diffusion stage, an exact integrator, so called exponential time differencing (ETD) method is applied if the diffusion coefficients are constant. Such an ETD method solves the diffusion stage equation exactly (by keeping the finite difference spatial discretization), so that both the positivity-preserving and energy stability are ensured. If the diffusion coefficients are nonlinear, we have to apply a similar idea as in the reaction stage: a predictor-corrector approach in the mobility approximation and a modified Crank-Nicolson algorithm for the chemical potential. 
  In either case, both the positivity-preserving and energy stability could be theoretically justified for the numerical solution in the diffusion stage. Finally, a combination of the numerical algorithms at both stages by the Strang splitting approach leads to a second-order accurate, structure preserving scheme for the original reaction-diffusion system.

The rest of this article is organized as follows. The energetic variational approach is reviewed in Section 2, for the reaction-diffusion systems with the detailed balance condition. Subsequently, the second-order operator splitting scheme is presented in Section 3. The positivity-preserving and energy stability analyses will be provided at each stage as well. Some numerical results will be presented in Section 4, to demonstrate the performance of the second order operator splitting scheme.


\section{Review of the energetic variational approach for reaction-diffusion systems} 
In this section, we briefly review the energetic variational approach for reaction-diffusion systems with detailed balance, which will be the foundation of the second order operator splitting scheme developed in the next section. We refer interested readers to \cite{liu2020structure, wang2020field} for more detailed descriptions.

The energetic variational approach (EnVarA) \cite{eisenberg2010energy, giga2017variational, liu2009introduction}, which is inspired by the seminal works of Rayleigh \cite{rayleigh1873note} and Onsager \cite{onsager1931reciprocal,onsager1931reciprocal2}, provides a systematic way to derive the dynamics of the system from a prescribed energy-dissipation law. In more details, an energy-dissipation law, which comes from the first and second law of thermodynamics, can be written as 
$$ 
  \frac{\dd}{\dd t} E^{\rm total} = - \triangle , 
$$
for an isothermal closed system, where $E^{\rm total}$ is the total energy, including both the kinetic energy $\mathcal{K}$ and the Helmholtz free energy $\mathcal{F}$, and $\triangle \geq 0$ is the energy dissipation rate which is equal to the entropy production in the process. The energy-dissipation law, along with the kinematics of employed variables, describe all the physics and the assumptions in the system. Starting with an energy-dissipation law, the EnVarA derives the dynamics of the systems through two variational principles, the Least Action Principle (LAP) and the Maximum Dissipation Principle (MDP). The LAP, which states the equation of motion for a Hamiltonian system can be derived from the variation of the action functional $\mathcal{A} = \int_{0}^T \mathcal{K} - \mathcal{F} \dd t$, with respect to the flow maps, gives a unique procedure to derive the conservative force for the system. 
In the MDP, variation of the dissipation potential $\mathcal{D}$, which equals to $\frac{1}{2}\triangle$ in the linear response regime, with respect to the rate (such as velocity), gives the dissipation force for the system. In turn, the force balance condition leads to the evolution equation to the system
\begin{equation*}
\frac{\delta \mathcal{D}}{\delta \x_t} = \frac{\delta \mathcal{A}}{\delta \x}.
\end{equation*}
In this formulation, the energy-dissipation law, along with the kinematics of state variables,describes all the physics and the assumptions for a given system.
The energetic variational approach has been successfully applied to build up many mathematical models \cite{giga2017variational}, including systems with chemical reactions \cite{wang2020field, wang2021two}; it has also provided a guideline of designing structure-preserving numerical schemes for systems with variational structures \cite{liu2020structure, liu2019lagrangian, liu2020variational}, etc.

\subsection{Reaction kinetics}

Consider a system with $N$ species $\{ X_1, X_2, \ldots X_N \}$ and $M$ reversible chemical reactions given by
\begin{equation}
\ce{ $\alpha_{1}^{l} X_1 + \alpha_{2}^{l}X_2 + \ldots \alpha_{N}^{l} X_N$ <=> $\beta_{1}^{l} X_1 + \beta_{2}^{l}X_2 + \ldots \beta_{N}^{l} X_N$}, \quad l = 1, \ldots, M.
\end{equation}
Denote ${\bm c} = (c_1, c_2, \ldots, c_N)^{\rm T}$, the concentrations of all species.  
The variable vector ${\bm c}$ satisfies the reaction kinetics
\begin{equation}\label{rk1}
\pp_t c_i  = \sum_{l=1}^M \sigma_{il} r_l({\bm c}),
\end{equation}
where $r_l({\bm c})$ is the reaction rate for $l-$the chemical reaction, and $\sigma_{il} = \beta^l_i - \alpha^l_i$ is the stoichiometric coefficients. From (\ref{rk1}), it is noticed that
\begin{equation}\label{conserv}
\frac{\dd}{\dd t} (\bm{e} \cdot c) = \bm{e} \cdot \bm{\sigma} \bm{r}(\bm{c}(t), t) = 0, \quad 
 \mbox{for}  \, \, \,  \bm{e} \in Ker(\bm{\sigma}^{\rm T}) .  
\end{equation}
In turn, one can define  $N - rank ({\bm \sigma})$ linearly independent conserved quantities for the reaction network. In the classical chemical kinetics, $r_l({\bm c})$ is determined by the law of mass action (LMA), which states that
the reaction rate is directly proportional to the product of the reactant concentrations, i.e., 
\begin{equation}\label{LMA}
r_l (\bm{c}) = k_{l}^+ {\bm c}^{{\bm \alpha}^l} -  k_{l}^- {\bm c}^{{\bm \beta}^l}, \quad 
   {\bm c}^{{\bm \alpha}^l}  = \prod_{i=1}^N c_i^{\alpha_i^l}, \quad {\bm c}^{{\bm \beta}^l}  = \prod_{i=1}^N c_i^{\beta_i^l} , 
\end{equation}
in which $k_{l}^+$ and $k_{l}^-$ are the forward and backward reaction constants for the $l$-th reaction.


The free energy of the system can be written as \cite{mielke2017non,wang2020field} 
\begin{equation}\label{free_energy_U}
\mathcal{F}[c_i] =  \int \sum_{i=1}^N \left( c_i (\ln c_i - 1) + c_i U_i \right) \dd \x,
\end{equation}
where the first part stands for the entropy, and $U_i$ is the internal energy associated with each species. In general, $U_i$ depends on ${\bm c}$ and $\x$, and the choice of $U_i$ determines the equilibrium of the system. We assume that $U_i$ is a constant throughout this paper.
Moreover, it has been shown that the reaction kinetics (\ref{rk1}) along with the law of mass action (\ref{LMA}) admits a Lyapunov function if there exists a strictly positive equilibrium point $\bm{c}_{\infty} \in \mathbb{R}^N_{+}$, satisfying  
\begin{equation}
  k_{l+} {\bm c}_{\infty}^{{\bm \alpha}^l} = k_{l-} {\bm c}_{\infty}^{{\bm \beta}^l} ,  
  \quad l  = 1, \ldots M.
\end{equation}
The condition is known as the \emph{detailed balance} condition. Within $\bm{c}_{\infty}$, one can define the Lyapunov function as
\begin{equation}\label{free_energy_c}
\mathcal{F}[c_i] = \sum_{i=1}^N c_i \left( \ln \left( \frac{c_i}{c_i^{\infty}}  \right) - 1 \right) . 
\end{equation}
It can be noticed that $c_i^{\infty}$ and $U_i$ are related through
\begin{equation}
 \sum_{i=1}^N \alpha_i^l (\ln c_i^{\infty} + U_i) = \sum_{i=1}^N \beta_i^l  (\ln c_i^{\infty} + U_i), \quad l = 1, \ldots, M.
\end{equation}

To transform the reaction kinetics into a variational frame, it is important to introduce another state variable ${\bm R} \in \mathbb{R}^M$, known as the reaction trajectory \cite{oster1974chemical, wang2020field}, or the extent of reaction \cite{de1927affinite, kondepudi2014modern}. The $l$-th component of ${\bm R}(t)$ corresponds to the number of $l$-th reaction that has happened by time $t$ in the forward direction. For any initial condition ${\bm c}(0) \in \mathbb{R}^N_{+}$, the value of ${\bm c}(t)$ can be represented in terms of ${\bm R}$ as the following formula 
\begin{equation}\label{c_R_1}
      \bm{c}(t) = \bm{c}(0) + \bm{\sigma} {\bm R}(t),  \quad 
      \mbox{${\bm \sigma} \in \mathbb{R}^{N \times M}$ is the stoichiometric matrix} .       
\end{equation}
This equation can be viewed as the kinematics of a reaction kinetics, which embodies the conservation properties (\ref{conserv}).
In particular, the positivity of ${\bm c}$ requires a constraint on ${\bm R}$: 
$$ 
  \bm{\sigma} {\bm R}(t) + \bm{c}(0) > 0 . 
$$
Subsequently, the reaction rate ${\bm r}$ can be defined as $\dot{\bm R}$, known as the reaction velocity \cite{kondepudi2014modern}. In the framework of the EnVarA, we can describe the reaction kinetics through the energy-dissipation law in terms of ${\bm R}(t)$ and $\dot{\bm R}$: 
\begin{equation}\label{ED_R}
\frac{\dd}{\dd t} \mathcal{F}[ {\bm c} ({\bm R})] = - \mathcal{D}_{\rm chem}[{\bm R},  \dot{\bm R}],
\end{equation}
where $\mathcal{D}_{\rm chem}[{\bm R}, \dot{\bm R}]$ is the rate of energy dissipation in the chemical reaction process. Unlike mechanical systems, the rate of energy dissipation for reaction kinetics may not be quadratic in terms of $\dot{\bm R}$, since the system is often far from equilibrium~\cite{beris1994thermodynamics, de2013non}. 
For a general nonlinear energy dissipation
\begin{equation}
\mathcal{D}_{\rm chem}[{\bm R}, \dot{\bm R}] = \left( {\bm \Gamma}({\bm R}, \dot{\bm R}), \dot{\bm R}  \right) = \sum_{l = 1}^M \Gamma_l ({\bm R}, \dot{\bm R})  \dot{R}_l \geq 0,
\end{equation}
since
\begin{equation}
  \frac{\dd}{\dd t} \mathcal{F} = \left(\frac{\delta \mathcal{F}}{\delta {\bm R}}, \dot{\bm R}   \right)  = \sum_{l=1}^M \frac{\delta \mathcal{F}}{\delta R_l}  \dot{R}_{l},
\end{equation}
one can specify
\begin{equation}\label{R1}
\Gamma_l ({\bm R}, \dot{\bm R}) = - \frac{\delta \mathcal{F}}{\delta R_l}.
\end{equation}
such that the energy-dissipation law (\ref{ED_R}) holds. Equation (\ref{R1}) is the reaction rate equation obtained by an energetic variational approach. It is interesting to notice that 
\begin{equation} \label{notation-mu-1} 
\frac{\delta \mathcal{F}}{\delta R_l} = \sum_{i=1}^N \frac{\delta \mathcal{F}}{\delta c_i} \frac{\delta c_i}{\delta R_l} =\sum_{i=1}^N \sigma_i^l \mu_i,
\end{equation}
which turns out to be the chemical affinity, and $\mu_i =  \frac{\delta \mathcal{F}}{\delta c_i}$ is the chemical potential of $i-$th species. The chemical affinity is the driving force of the chemical reaction \cite{de1927affinite, de1936thermodynamic, kondepudi2014modern}, and the dissipation makes a connection between the reaction rate $\dot{\bm R}$ and the chemical affinity. 
A typical choice of $mathcal{D}_{chem}[{\bm R},  \dot{\bm R}]$ is given by 
\begin{equation}
  \mathcal{D}_{chem}[{\bm R}, \dot{\bm R}] = \sum_{l = 1}^M \dot{R}_l \ln \Big( \frac{\dot{R}_l}{\eta_l( {\bm c}({\bm R}))}  + 1 \Big).
\end{equation}
One can derive the law of mass action by taking $\eta_l( {\bm c}({\bm R})) = k_l^-{\bm c}(R)^{{\bm \beta}_l}$.
Since $\dot{R}_l \approx 0$ near an equilibrium, we see that 
\begin{equation}
  \mathcal{D}_{chem}[{\bm R}, \dot{\bm R}] = \sum_{l = 1}^M \dot{R}_l \ln \left( \frac{\dot{R}_l}{\eta_l( {\bm c}({\bm R}))}  + 1\right) \approx  \sum_{i=1}^N\frac{1}{\eta_l ({\bm c} ({\bm R}) )} \dot{R}_l^2, \quad R_l \ll 1.
\end{equation}
In turn, the energy-dissipation law (\ref{ED_R}) becomes an $L^2$-gradient flow in terms of ${\bm R}$.

\begin{remark}
  In Onsager's celebrated paper \cite{onsager1931reciprocal}, instead of writing $\mathcal{D}_{\rm chem}[{\bm R},  \dot{\bm R}]$ as a non-quadratic form, it was argued that chemical affinity (\ref{notation-mu-1}) can be linearized near the equilibrium, i.e. ${\bm c}^0$ is closed to ${\bm c}^{\infty}$ and ${\bm R}(t)$ is close to zero. 
\end{remark}

The reaction kinetics can be viewed as a generalized gradient flow, with a nonlinear mobility in terms of the reaction trajectory. Hence, it is expected that the numerical techniques for $L^2$-gradient flows can be  applied to reaction kinetics.

\subsection{Reaction-diffusion systems}
One can extend the energetic variational formulation for reaction kinetics to reaction-diffusion system with detailed balance, which is the foundation of the operator splitting scheme developed in the next section. For a reaction-diffusion system with $N$ species and $M$ reactions, the concentration ${\bm c} \in \mathbb{R}^N$ satisfies the kinematics
\begin{equation}\label{Kin_1}
 \pp_t c_i + \nabla \cdot (c_i \uvec_i) = \left( {\bm \sigma} \dot{\bm R} \right)_i, \quad i = 1, 2, \ldots N , 
\end{equation}
where $\uvec_i$ is the average velocity of each species by its own diffusion,  ${\bm R} \in \mathbb{R}^M$ represents various reaction trajectories involved in the system, with ${\bm \sigma} \in \mathbb{R}^{N \times M}$ being the stoichiometric matrix as defined in section 2.1. The quantities $\uvec_i$ and ${\bm R}$ can be obtained through an energy-dissipation law \cite{biot1982thermodynamic, wang2020field} 
\begin{equation}\label{ED_RD}
    \frac{\dd}{\dd t} \mathcal{F}[{\bm c}({\bm R})]  =  - (2 \mathcal{D}_{\rm mech} + \mathcal{D}_{\rm chem}), 
\end{equation}
which leads to a reaction-diffusion equation.
Here  $\mathcal{F}[{\bm c}]$ is the free energy given by (\ref{free_energy_U}), and 
$\mathcal{D}_{\rm mech}$ and $\mathcal{D}_{\rm chem}$ are dissipations for the mechanical and reaction parts, respectively. One key point is that the reaction and diffusion parts of the system dissipate the same free energy. To derive the reaction diffusion equation (\ref{RD_1}), $\mathcal{D}_{\rm mech}$ could be taken as
$$
  2 \mathcal{D}_{\rm mech} =  \int_{\Omega} \sum_{i=1}^N \eta_i(c_i) |\uvec_i|^2 \dd \x,  \quad 
  \mbox{$\eta_i$ is the friction coefficient} , 
$$ 
and $\mathcal{D}_{\rm chem}$ could be taken as
\begin{equation*}
  \mathcal{D}_{\rm chem} =  \int_{\Omega} \sum_{l=1}^M  \dot{R}_l \ln \left(  \frac{  \dot{R}_l}{\eta({{\bm c}({\bm R})})} \right) \dd \x.
\end{equation*}

The energetic variational approach could be applied to the reaction and diffusion parts, respectively, so that the ``force balance equation'' is obtained for the chemical and mechanical subsystems. 
Formally, a direct computation implies that  
\begin{equation}
\frac{\dd}{\dd t} \mathcal{F}[{\bm c}] = \sum_{i=1}^N \left( c_i \nabla \mu_i, \uvec_i  \right) + \sum_{l=1}^M \left( \sum_{i=1}^N \sigma_{i}^l \mu_i,  \dot{R}_l  \right),
\end{equation}
which in turn gives  
\begin{equation}
  \begin{cases}
    & \eta_i(c_i) \uvec_i =  - c_i \nabla \mu_i, \quad i = 1, 2, \ldots N , \\
    &  \ln \left(  \frac{ \dot{R}_l}{\eta({{\bm c}({\bm R})})} \right) = - \sum_{i=1}^N \sigma_i^l \mu_i, \quad l = 1, \ldots, M .\\
  \end{cases}
\end{equation}
In particular, a linear reaction-diffusion system can be obtained by choosing $\eta_i(c_i) = \frac{1}{D_i}c_i$: 
\begin{equation}
\pp_t c_i = D_i \Delta c_i + (\sigma \pp_t {\bm R})_i ,  \quad 
\mbox{$({\bm \sigma} \pp_t {\bm R})_i$ is the reaction term} . 
\end{equation}
Other choices of $\eta_i(c_i)$ can result in some porous medium type nonlinear diffusion equation \cite{liu2019lagrangian}
\begin{equation}\label{RD_Final}
\pp_t c_i = \nabla \cdot (D(c_i) \nabla c_i) + (\sigma \pp_t {\bm R})_i,
\end{equation}
where $D(c_i) = \frac{c_i}{\eta(c_i)}$ is the concentration-dependent diffusion coefficient.

In this formulation, the reaction part is reformulated in terms of reaction trajectories $R$,  and the reaction and diffusion parts impose different dissipation mechanisms for the same physical energy.

\section{The second-order operator splitting scheme}
In the section, we construct a second-order operator splitting scheme to a reaction-diffusion system based on the energetic variational formulation outlined in the last section, in which the numerical discretization for the reaction part is applied to the reaction trajectory $R$ in the reaction space, 
while the numerical method for the diffusion part is designed to the concentration ${\bm c}$ in the species space. 
To illustrate the idea, we focus on a case with  one reversible detailed balance reaction, given by
\begin{equation}\label{reaction_sec3}
  \ce{\alpha_1 X_1 + \ldots \alpha_r X_r <=>[k^+_1][k^-_1] \beta_{r+1} X_{r+1} + \ldots \beta_N X_N},
  \end{equation}
  where $k^+_1$ and $k^-_1$ are constants. Moreover, we assume that the reaction-diffusion system satisfies the energy-dissipation law (\ref{ED_RD}). Numerical schemes for systems involving multiple reversible reactions can be constructed in the same manner. 

To simplify the numerical description, the reaction-diffusion equation (\ref{RD_Final}) can be rewritten as
\begin{equation}
	{\bm c} = \mathcal{A} {\bm c} + \mathcal{B} {\bm c},
\end{equation}
where $\mathcal{A}$ is a reaction operator and $\mathcal{B}$ a diffusion operator. Throughout this section, the computational domain is taken as $\Omega = (0,1)^3$ with a periodic boundary condition, and $\Delta x = \Delta y = \Delta z = h = \frac{1}{N_0}$ with $N_0$ being the spatial mesh resolution throughout this section; a computational domain with other boundary condition or numerical mesh could be analyzed in a similar fashion. In addition, the discrete free energy is defined as follows, with the given spatial discretization: 
\begin{equation} 
\mathcal{F}_h ({\bm c}): = \langle  \sum_{i=1}^N \left( c_i (\ln c_i - 1)  + c_i U_i \right), {\bf 1} \rangle,   
     \label{splitting-R}     
  \end{equation} 
where $\langle f , g \rangle = h^3 \sum_{i,j,k=0}^{N_0-1} f_{i,j,k} g_{i,j,k}$ denotes the discrete $L^2$ inner product.

Following the second-order Strang splitting formula ${\bm c}^{n+1} = e^{\frac{1}{2} \Delta t \mathcal{A}} e^{\Delta t \mathcal{B}} e^{\frac{1}{2} \Delta t \mathcal{A}} {\bm c}^n$ \cite{strang1968construction}, the numerical solution ${\bm c}^{n+1}$ can be obtained through three stages. Given ${\bm c}^n$ with ${\bm c}_{i,j,k}^n \in \mathbb{R}^N_{+}$, we update ${\bm c}^{n+1}$ via the following three stages. 

\noindent 
{\bf Stage 1.} \, First, we set ${\bm c}_0 = {\bm c}^n$ and solve the reaction trajectory equation, subject to the initial condition $R^n = 0$, with a second-order, positivity-preserving, energy-stable scheme,  with the temporal step-size $\dt / 2$. An intermediate numerical profile is updated as 
\begin{equation}
{\bm c}^{n+1, (1)} = {\bm c}^{n} + {\bm \sigma} R^{n+1, (1)}.
\end{equation}

\noindent 
{\bf Stage 2.} \, Starting with the intermediate variable ${\bm c}^{n+1, (1)}$, we solve the diffusion equation $\pp_t {\bm c} = \mathcal{B} {\bm c}$ by a second-order,  positivity-preserving and energy-stable scheme with the temporal step-size $\dt$ to obtain ${\bm c}^{n+1, (2)}$.

\noindent
{\bf Stage 3.} \, We set ${\bm c}_0 = {\bm c}^{n+1, (2)}$ and repeat the stage 1, i.e., solving the reaction trajectory equation, subject to the initial condition $R^n = 0$  with the temporal step-size $\dt / 2$ to obtain  $R^{n+1, (2)}$. The numerical solution at $t^{n+1}$ is updated as 
\begin{equation}
  {\bm c}^{n+1} = {\bm c}^{n+1, (2)} + {\bm \sigma} R^{n+1, (2)}.
  \end{equation}

More details of the numerical algorithms at each stage will be provided in the following subsections.


\subsection{Second-order algorithm for reaction kinetics}
We first develop a second order algorithm for the reaction stage, which only needs to be constructed in a point-wise sense.
The discrete free energy can be reformulated in terms of $R$ at each mesh point, denoted by 
\begin{equation}
F(R) = \sum_{i=1}^N c_i(R) ( \ln c_i(R) - 1 ) + c_i (R) U_i.
\end{equation}
For simplicity of presentation, we omit the grid index throughout this subsection. Following the earlier discussions, for a given initial condition ${\bm c}^0$, the reaction trajectory equation is given by
\begin{equation}\label{eq_R_section3}
  \begin{cases}
   & \ln \left( \frac{R_t}{\eta({\bm c}(R))} + 1 \right) = - \mu(R) ,  \\   
   & \mu(R) =  \frac{\delta F}{\delta R} = \sum_{i=1}^N \sigma_i \mu_i(c_i(R)) , \\
  \end{cases}
\end{equation}
where $\eta({\bm c}(R))$ is the nonlinear mobility that takes the form $\eta({\bm c}(R)) = k_1^- \prod_{i = r+1}^N c_i^{\beta_i}$, ${\bm c}(R) = {\bm c}^0 + {\bm \sigma} R$ with $ {\bm \sigma} = ( - \alpha_1, -\alpha_2, \ldots, - \alpha_r, \beta_1, \beta_2, \ldots, \beta_N)^{\rm T}$ is the stoichiometric vector, and $ \mu_i(c_i) = \ln c_i + U_i$ is the chemical potential associated with $i$-species.
Similar to an $L^2-$gradient flow, a second-order algorithm for the reaction trajectory equation (\ref{eq_R_section3}) can be constructed through a Crank-Nicolson type discretization
\begin{equation}
\ln \left( \frac{R^{n+1} - R^n}{\eta({\bm c}(R^*)) \Delta t} + 1 \right) = - \mu^{n+1/2},
\end{equation}
where $\mu^{n+1/2}$ is a suitable approximation to the chemical affinity, $F'(R)$, at $t_{n+1/2}$, $R^*$ is an approximation to $R^{n+1/2}$, which needs to be independent on $R^{n+1}$. The primary difficulty is focused on the construction of $R^*$ and $\mu^{n+1/2}$, to ensure the unique solvability, as well as the positivity of $R^{n+1} - R^n + \eta({\bm c}(R^*)) \Delta t$ and ${\bm c}(R^{n+1})$.

 First, we use a first-order scheme to obtain a rough ``guess'' to $R^{n+1}$, denoted by $\widehat{R}^{n+1}$, as a numerical solution to
\begin{equation}\label{1st_R}
\ln \left( \frac{\widehat{R}^{n+1} - R^n}{\eta({\bm c}(R^n)) \dt} + 1 \right) = \sum_{i=1}^N \sigma_i \mu_i (\widehat{R}^{n+1}) , 
\end{equation}
in the admissible set. This first-order scheme was proposed in \cite{liu2020structure}, while the unique solvability and the positivity preserving property have been proved. With $\widehat{R}^{n+1}$ at hand, we introduce $R^* = (R^n + \widehat{R}^{n+1}) / 2$. Although (\ref{1st_R}) corresponds to a first order truncation error, we see that $\widehat{R}^{n+1}$ is a second order approximation to $R^{n+1}$, locally in time, due to the $\dt$ term in the denominator. In turn, $R^*$ becomes a second order approximation to $R^{n+1/2}$. To approximate $\left(\frac{\delta \mathcal{F}}{\delta R} \right)^{n+1/2}$, we apply the idea of discrete variational derivative method \cite{du1991numerical, furihata2010discrete}. More specifically, the following function is introduced 
\begin{equation}
	\phi(p, q) = 
\begin{cases}  
  & \frac{ F(p) - F(q) }{p - q},    \, \,  p \neq q , \\
  &  F'(p), \quad \quad  p = q      ,      \\
\end{cases}
\end{equation}
as a second-order approximation to $F'(\frac{p + q}{2})$. In fact, it is also known as the discrete variation of $F(R)$ \cite{furihata2010discrete}.

With the combined arguments, the second-order algorithm is constructed as
\begin{equation}\label{2nd_DVD}
	\begin{cases}
& \ln \left( \frac{R^{n+1} - R^n}{\eta( {\bm c} (\widehat{R}^{n+1/2})) \Delta t} + 1   \right) = - \mu_R^{n+1/2}, \quad \widehat{R}^{n+1/2} = \frac{1}{2} (R^n + \widehat{R}^{n+1}) , \\
& \mu^{n+1/2}_R = 
\phi(R^{n+1}, R^n)+ \Delta t \sum_{i=1}^N \sigma_i (\mu_i(R^{n+1}) - \mu_i(R^n)) . \\  
	\end{cases}
\end{equation}
The term $\Delta t \sum_{i=1}^N \sigma (\mu_i(R^{n+1}) - \mu_i(R^n))$ is added for the theoretical analysis the positivity-preserving property. This $O (\dt^2)$ term is artificial, and it will not effect the second order accuracy in the temporal discretization. 

This algorithm can be reformulated as an optimization problem
\begin{equation}\label{MinProblem_J}
  \begin{cases}
& R = \mathop{\arg\min}_{R \in \mathcal{V}_n} J_n(R), \\
& J_n(R) =   
\Psi_n(R, R^n) +  \int_{R^n}^R \phi(s, R^n) \dd s +  \lambda (\Delta t F (R) - (\gamma^n, R)), \\
& \mathcal{V}_n = \left\{ R ~|~ c_i(R) > 0, \quad R - R^n + \eta(c(\hat{R}^{n+1/2})) \dt > 0 \right\} , 
  \end{cases}
\end{equation}
where $\gamma^n = \Delta t \sum_{i=1}^N \sigma_i \mu_i (c_i (R^n))$, and 
\begin{equation}
  \Psi_n(R, R^n) = (R - R^n + \eta( {\bm c} (\hat{R}^{n+1/2}) ) \dt ) \ln \left(  \frac{R - R^n}{ \eta( {\bm c} (\hat{R}^{n+1/2}) ) \dt} + 1 \right) - (R - R^n) 
\end{equation}
is a function that measures the ``distance'' between $R$ and $R^n$. An explicit form of $J_n(R)$ is not available. On the other hand, we can prove that $J_n(R)$ admits a unique minimizer in the admissible set. More precisely, the following theorem is valid. 

\begin{thm}\label{thm1}
Given ${\bm c}^n > 0$ and $R^n = 0$, there exists a unique solution $R^{n+1}$  for the minimization problem (\ref{MinProblem_J}), which turns out to be the unique solution for the numerical scheme (\ref{2nd_DVD}), with ${\bm c} (R^{n+1}) > 0$  and $R^{n+1} + \eta(c(\widehat{R}^{n+1/2})) \Delta t > 0$. Therefore, the numerical scheme is well-defined.
\end{thm}


To facilitate the proof of this result, the following smooth functions are introduced, for fixed $a > 0$:
\begin{equation}\label{def_G}
  \begin{aligned}
       & G_{a}^1 (x) = \frac{x \ln x - a \ln a}{x - a}  , \\  
       & G_a^0 (x) = \int_{a}^x G_a^{1} (s) \dd s , \\  
       & G_a^2(x) = (G_a^0)'' (x) = (G_a^1)' = \frac{x -a + a(\ln a - \ln x) }{(x - a)^2}  .  \\    
   \end{aligned}
\end{equation}
By a direct calculation, it is straightforward to prove the following results, which will be used in the proof of Theorem \ref{thm1}.

\begin{lem} \label{lemam1}
For any fixed $a > 0$, we have: (1) $G_a^2(x) \geq 0$ for any $x > 0$; (2) $G_a^0(x)$ is convex in terms of $x$; (3) There exists $\xi$ between $a$ and $x$ such that $(G_a^1)'(x) = \frac{1}{\xi}$; (4) Since $G_a^1(x)$ increases in terms of $x$, we have $G_a^1(x) \leq G_a^1(a)$ for any $0 < x \leq a$.
\end{lem}

Now we can proceed into the proof of Theorem~\ref{thm1}. 
\begin{proof}
Recall the minimization problem~\eqref{MinProblem_J},   
and it is clear that $J_n (R)$ is a strictly convex function over $\mathcal{V}_n$. We only need to prove that the minimizer of $J_n(R)$ over $\mathcal{V}$ could not occur on the boundary of $\mathcal{V}$, so that a minimizer corresponds to a numerical solution of (\ref{2nd_DVD}) in $\mathcal{V}_n$.

The following closed domain is considered in the analysis: 
  \begin{eqnarray} 
	\mathcal{V}_{\delta} = \left\{ R ~|~ c_i(R) \ge \delta, \quad R - R^n + \eta(c(\hat{R}^{n+1/2})) \dt \ge \delta \right\} \subset \mathcal{V} . 
	\label{Field-positive-2} 
  \end{eqnarray}   
A careful calculation indicates that, for any $R \in \mathcal{V}_{\delta}$, the following bounds are satisfied
  \begin{equation}
  \max \frac{1}{\beta_i} (\delta  - c_i^0) \le  R \le \min \frac{1}{\alpha_i} (c_i^0 - \delta), \quad  R \ge R^n -  \eta(c(\hat{R}^{n+1/2})) \dt + \delta,
  \end{equation}
i.e., $\mathcal{V}_{\delta}  = [ \max \tfrac{1}{\beta_i} (\delta  - c_i^0), \min \tfrac{1}{\alpha_i} (c_i^0 - \delta) ]$ or $\mathcal{V}_{\delta}  = [ R^n -  \eta(c(\hat{R}^{n+1/2})) \dt + \delta, \min \tfrac{1}{\alpha_i} (c_i^0 - \delta) ]$.  Since $\mathcal{V}_{\delta}$ is a bounded, compact set, there exists a (may not unique) minimizer of $J_n (R)$ over $\mathcal{V}_{\delta}$. Moreover, we have to prove that, such a minimizer could not occur on the boundary points in $\mathcal{V}_\delta$, if $\delta$ is sufficiently small, by using the singular property of logarithmic function approaches to $0$. 

  Without loss of generality, the minimization point is assumed to be $R^*= R^n - \eta(c( \hat{R}^{n+1/2}) \Delta t + \delta$. 
A direct calculation gives 
  \begin{equation} 
    \begin{aligned}
   J'_n (R) \mid_{R=R^*} & =  
   \ln \delta + \phi(R^*, R^n) + \dt ( \mu(R^*) - \gamma^n)  . 
    \end{aligned}
    \label{Field-positive-4}    
  \end{equation} 
Next we show that $\phi(R^*, R^n) + \dt ( \mu(R^*) - \gamma^n)$ is bounded, so that we can choose $\delta$ sufficiently small with
  \begin{eqnarray} \label{Ineq_R}
	J'_n (R) \mid_{R=R^*} < 0, 
  \end{eqnarray}  
which leads to a contradiction since there will be $R^{*'} =  R^n - \eta(c(\hat{R}^{n+1/2})) \Delta t + \delta + \delta' \in V_\delta$ such that
\begin{eqnarray} 
	J_n (R^{*'}) < J_n (R^{*}) .  \label{Field-positive-14} 
\end{eqnarray} 
To derive a bound for $\phi(R^*, R^n) + \dt ( \mu(R^*) - \gamma^n)$, we notice that 
\begin{equation}
    \phi(R^*, R^n) =  \sum_{i=1}^N \sigma_i G_{c_i^n}^1(c_i^0 + \sigma_i R^*) +  \sum_{i = 1}^N\sigma_i (U_i - 1),
\end{equation}
where $c_i^n = c_i^0 + \sigma_i R^n$ and $\sum_{i = 1}^N\sigma_i (U_i - 1)$ is a constant. Since $G_a^1(x)$ is an increasing function of $x$ for any $a > 0$, the following inequality is valid: 
  \begin{equation}
    \begin{aligned}
    G_{c_i^n}^1(c_i^0 + \sigma_i R^*) & =  G_{c_i^n}^1(c_i^0 + \sigma_i (R^n -  \hat{\eta}^{n*} \Delta t  - \delta )  )   = G_{c_i^n}^1(c_i^0 + \sigma_i R^n -  \sigma_i (\hat{\eta}^{n*} \Delta t  - \delta )  )  \\
    & \geq G_{c_i^n}^1(c_i^0 + \sigma_i R^n) = \ln c_i^n + 1, \quad \sigma_i < 0 , 
    \end{aligned}
  \end{equation}
in which $\delta$ is sufficiently small such that $\hat{\eta}^{n*} \Delta t  - \delta > 0$. Similarly, we have
\begin{equation}
  G_{c_i^n}^1(c_i^0 + \sigma_i R^*) \leq G_{c_i^n}^1(c_i^0 + \sigma_i R^n) = \ln c_i^n + 1, \quad \sigma_i > 0,
\end{equation}
with $\delta$ sufficiently small such that $\hat{\eta}^{n*} \Delta t  - \delta > 0$. Hence,
\begin{equation}\label{es_1}
  \begin{aligned}
  \phi(R^*, R^n) &  =  \sum_{i=1}^N \sigma_i G_{c_i^n}^1(c_i^0 + \sigma_i R^*) +  \sum_{i = 1}^N\sigma_i (U_i - 1)  \\
  & \leq \sum_{i=1}^N \sigma_i \ln c_i^n + C_0,  
  \quad  C_0 = \sum_{i = 1}^N\sigma_i U_i  . 
  \end{aligned}
\end{equation}
Following the same argument, the following inequality could be derived: 
\begin{equation}\label{es_2}
  \begin{aligned}
\mu(R^*) & = \sum_{i=1}^N \sigma_i \mu_i (c_i + \sigma_i R^*) = \sum_{i=1}^N \sigma_i \ln (c_i^0 + \sigma_i R^*) + \sum_{i=1}^N \sigma_i U_i \\
         & \leq \sigma_i \ln (c_i^n) + C_0, \\
  \end{aligned}
\end{equation}
since $\ln x$ is an increasing function of $x$. A combination of (\ref{es_1}) and (\ref{es_2}) gives 
\begin{equation}
  J'_n (R) \mid_{R=R^*}  \leq  \ln \delta + C_1, 
\end{equation}
where $C_i = (1 + \dt) \sum_{i = 1}^N \sigma_i \ln (c_i^n) + (1 + \dt) C_0 - \dt \gamma^n$ is a constant. So we can choose $\delta$ small enough such that $J'_n (R) \mid_{R=R^*} < 0$, which leads to the contradiction inequality (\ref{Field-positive-14}).


  Using similar arguments, if $R^*= \min \frac{1}{\alpha_i} (c_i^0 - \delta) = \frac{1}{\alpha_q} (c_q^0 - \delta)$, we can prove that 
  \begin{equation}
    J'_n (R) \mid_{R=R^*}  \geq C_2 + \dt ( - \alpha_q) \ln \delta.
  \end{equation}
Then $\delta$ can be chosen to be sufficiently small such that $J'_n (R) \mid_{R=R^*}  > 0$, which leads to a contradiction. Meanwhile, if $R^* =  \max \frac{1}{\beta_i} (\delta - c_i^0)$, we will have $ J'_n (R) \mid_{R=R^*}  < 0$. 
  
  As a result, the global minimum of $J_n (R)$ over $V_{\delta}$ could only possibly occur at an interior point, if $\delta$ is sufficiently small. In turn, there is a minimizer $R^* \in (V_{\delta})^{\mathrm{o}}$, in the interior region of $V_{\delta}$, of $J_n(R^*)$, so that $J'_n (R) =0$. In other words, $R^*$ has to be the numerical solution of \eqref{2nd_DVD}, provided that $\delta$ is sufficiently small. Therefore, the existence of a ``positive" numerical solution is proved.  In addition, since $J (R)$ is a strictly convex function over $V$, the uniqueness of this numerical solution follows from a standard convexity analysis. The proof of Theorem~\ref{thm1} is finished.  
  \end{proof} 


  The energy stability of the numerical scheme (\ref{2nd_DVD}) is stated below. 
	\begin{thm}
    \label{Field-energy stability-A} 
  For a given $R^n$, the numerical solution $R^{n+1}$ to~\eqref{2nd_DVD} satisfies the energy-dissipation estimate
    \begin{equation} 
  F (R^{n+1})  \le F (R^{n}),   \quad \mbox{at a point-wise level}. \label{Field-energy-A-0} 
    \end{equation} 
    \end{thm}

    \begin{proof}
Multiplying both side of (\ref{2nd_DVD}) by $R^{n+1} - R^n$ and rearranging terms yields
  \begin{equation}\label{dE}
	\begin{aligned}
\frac{F(R^{n+1}) - F(R^n)  }{\Delta t} = &  - \frac{R^{n+1} - R^n}{\Delta t} \ln \left( \frac{R^{n+1} - R^n}{\eta( {\bm c} (\widehat{R}^{n+1/2})) \dt } + 1   \right) \\
&  -  \sum_{i=1}^N \sigma_i(\mu_i^{n+1} - \mu_i^n) (R^{n+1} - R^n)  \\ 
\leq &  - \frac{R^{n+1} - R^n}{\Delta t} \ln \left( \frac{R^{n+1} - R^n}{\eta( {\bm c} (\widehat{R}^{n+1/2})) \dt } + 1   \right)  \leq 0 . 
	\end{aligned}
\end{equation}
In the derivation of the above inequality, the following fact has been used: 
\begin{equation}
  \sigma_i(\mu_i^{n+1} - \mu_i^n) (R^{n+1} - R^n) = \sigma_i (\ln (c_i^0 + \sigma_i R^{n+1}) - \ln (c_i^0 + \sigma_i R^{n})) (R^{n+1} - R^n) \geq 0 , 
\end{equation}
which comes from the monotonic property of the logarithmic function.
\end{proof} 

\begin{remark}
Without the additional term $\Delta t \sum_{i=1}^N \sigma_i (\mu_i(R^{n+1}) - \mu_i(R^n))$, the discrete energy dissipation law (\ref{dE}) is an exact time discretization to the continuous energy-dissipation law, which is the advantage of the discrete variational derivative method.    It is crucial to add this term to establish the positivity-preserving property of the numerical solution in the admissible set. Also see the related numerical analysis for the Cahn-Hilliard gradient flow with Flory-Huggins energy potential~\cite{chen19b, dong20b, dong19a, dong20a}, the Poisson-Nernst-Planck (PNP) system~\cite{liu2020positivity, Qian2021a}, etc. 
\end{remark}

\begin{remark} 
There have been extensive works of second order accurate, energy stable numerical schemes to various gradient flows, based on either modified Crank-Nicolson~\cite{baskaran13a, baskaran13b, diegel16,  guo16, hu09, shen12} or BDF2~\cite{LiW18, yan18} approach. Meanwhile, most existing works are multi-step methods, since a multi-step approximation to the concave terms is usually needed to ensure both the unique solvability and energy stability. However, for the operator splitting method, a single step, second order approximation has to be accomplished at each stage, so that these standard approach is not directly available. To overcome this difficulty, we construct a numerical profile $\widehat{R}^{n+1}$, a local-in-time second order approximation of $R$ at time step $t^{n+1}$, so that a multi-step approximation to the mobility function is avoided. In addition, the fact that the physical energy does not contain any concave part enables one to derive a single step, modified Crank-Nicolson method, while preserving the energy stability. 
\end{remark}

\subsection{Second-order schemes in the diffusion stage}
In this subsection, we present two positivity-preserving and energy-stable numerical algorithms for linear and nonlinear diffusion processes, respectively, which could be used in the diffusion stage. 
In particular, the cross-diffusion is not considered, so that the $N$ diffusion equations of $c_i$ are fully decoupled. Therefore, we only need to construct numerical algorithms for a diffusion equation
\begin{equation}\label{DF1}
\rho_t = \nabla \cdot (D(\rho, \x) \nabla \rho), \quad 
\mbox{$D(\rho, \x)$ is the diffusion coefficient} . 
\end{equation}
In fact, this diffusion equation satisfies an energy-dissipation law
\begin{equation}\label{ED_Law}
\int \rho \ln \rho +  C \rho \dd \x = - \int \mathcal{M}(\rho, \x) |\nabla \mu|^2 \dd \x, 
\end{equation}
where $\mathcal{M}(\rho, \x) = D(\rho, \x) \rho$ is known as the mobility, $C$ is an arbitrary constant, $\nabla \mu = \nabla (\ln \rho)$ turns out to be the gradient of the chemical potential $\mu = \rho \ln \rho + C \rho + 1$. With a careful spatial discretization, the discrete energy is defined as 
\begin{equation} 
  \mathcal{F}_h (\rho) := \langle \rho \ln \rho +  C \rho , {\bf 1} \rangle . 	\label{energy-phase-discrete-1} 
\end{equation} 


\subsubsection{An ETD scheme for a linear diffusion}
We first consider a linear diffusion with a constant coefficient, given by
\begin{equation}\label{L_DE}
	\rho_t = \mathcal{L} \rho, \quad \mathcal{L} = D \Delta , \, \, \, D > 0, 
\end{equation}
subject to the periodic boundary condition. Of course, the solution of linear diffusion equation~(\ref{L_DE}) satisfies the following maximum principle:
\begin{equation}
\max_{\Omega} \rho(\x , t) \leq \max_{\Omega} \rho(\x, 0), \quad \min_{\Omega} \rho(\x, t) \geq \min_{\Omega} \rho(\x, 0), \quad \forall t > 0.
\end{equation}

An easy way to obtain a high-order scheme to a linear diffusion equation is to apply the exponential time differencing (ETD) method \cite{cox2002exponential, kassam2005fourth}, which is indeed exact in time. More precisely, we can introduce the spatial discretization to (\ref{L_DE}) by the standard centered difference method, which leads to
\begin{equation}\label{Heat_spatial_d}
\pp_t \rho = \mathcal{L}_h \rho . 
\end{equation}
Integrating the above equation over a single time step from $t = t_n$ to $t_{n+1}$, we get
\begin{equation}\label{ETD_Heat}
 \rho^{n+1} = e^{\mathcal{L}_h \Delta t} \rho^n,
\end{equation}
which is known as the ETD scheme~\cite{cox2002exponential}. 

Due to the discrete maximum principle \cite{du2020maximum}, the following positivity-preserving property is obvious.
\begin{thm}  \label{Heat-positivity} 
  Given $\rho^n$, with $\rho_{i,j,k}^n > 0$, $0 \le, i,j,k \le N_0-1$, there exists a unique solution $\rho^{n+1}$ for the numerical scheme~\eqref{ETD_Heat}, with discrete period boundary condition, with $\rho_{i,j,k}^{n+1} > 0$, $0 \le i,j,k \le N_0 -1$.
\end{thm} %
 With the positivity-preserving and unique solvability for the numerical scheme~\eqref{ETD_Heat}, it is straightforward to prove an unconditional energy stability. 
	\begin{thm}
	\label{Heat-energy stability} 
For the numerical solution~\eqref{ETD_Heat}, we have 
	\begin{equation} 
    \mathcal{F}_h (\rho^{n+1}) \le \mathcal{F}_h (\rho^n),   \label{heat-energy-0} 
	\end{equation} 
so that $\mathcal{F}_h (\rho^n) \le \mathcal{F}_h (\rho^0_h)$, an initial constant. 
	\end{thm}
	
\begin{proof} 
Taking a discrete inner product with~\eqref{Heat_spatial_d} by $\ln \rho$ gives 
\begin{equation} 
  \langle \tfrac{\dd}{\dd t}\rho , \ln \rho \rangle = - \langle \nabla_h \rho ,  \nabla_h ( \ln \rho )  \rangle . 
  \label{heat-energy-1} 
\end{equation}   
By a direct calculation, we have
\begin{equation} 
   \frac{d}{dt} \mathcal{F}_h ( \rho) = \langle \tfrac{\dd}{\dd t}\rho , \ln \rho \rangle  = - \langle \nabla_h \rho ,  \nabla_h ( \ln \rho )  \rangle \leq 0,  \label{heat-energy-2} 
\end{equation}     
where the last inequality is due to the monotone property of the logarithmic function. This completes the proof. 
\end{proof} 

In fact, such a stability is available for not only $\mathcal{F}_h (\rho)$ given by~\eqref{energy-phase-discrete-1}, but also for all the convex energies. The following estimate could be derived using similar techniques. 
	\begin{cor}
	\label{Heat-energy stability-general} 
For the numerical solution~\eqref{Heat_spatial_d}, we have $\mathcal{F}_h (\rho^{n+1}) \le \mathcal{F}_h (\rho^n)$ for any $n \ge 0$, and $\mathcal{F}_h (\rho)$ taking a form of 
	\begin{equation*} 
    \mathcal{F} (\rho) = \langle F (u) , {\bf 1} \rangle ,   \quad \mbox{in which $F$ is a convex function of $\rho$, for $\rho >0$} .  \label{energy-phase-discrete-general} 
	\end{equation*} 
	\end{cor}

\subsubsection{Second-order scheme for a nonlinear diffusion equation} 
The ETD scheme is not suitable for nonlinear diffusion equations. 
The construction of a second-order accurate, positivity-preserving and energy stable scheme for a generalized nonlinear diffusion equation has always been very challenging. Here we present a general approach to achieve this goal.
For simplicity of presentation, it is assumed that the diffusion coefficient $D(\rho)$ depends only explicitly on $\rho$.  The case of $\x$-dependent coefficients could be handled in a similar manner. 

The idea is quite similar to the scheme (\ref{2nd_DVD}) in the reaction stage. First, we need a rough guess $\hat{\rho}^{n+1}$, which has to be point-wise positive, as 
a second order temporal approximation to $\rho^{n+1}$. 
The simplest way to obtain such a rough guess $\hat{\rho}^{n+1}$ is to use the classical semi-implicit scheme
\begin{equation} \label{semi_imp}
  \frac{\hat{\rho}^{n+1} - \rho^n}{\dt} = \nabla_h \cdot \left( \mathcal{A}_h [D (\rho^n)]  
   \nabla_h \hat{\rho}^{n+1,(2)}  \right),
\end{equation}
where $\nabla_h$ and $\nabla_h \cdot$ stand for the discrete gradient and the discrete divergence respectively, $\mathcal{A}_h [D (\rho^n)]$ is a spatially averaging operator introduced to obtain the value of $D (\rho^n)$ at staggered mesh points.  As proved in a recent work, the semi-implicit scheme (\ref{semi_imp}) satisfies the following uniquely solvable and positivity-preserving properties. 

\begin{prop}  \cite{liu2020structure} \label{Heat-positivity-2} 
  Given $\rho^n$, with $\rho_{i,j,k}^n > 0$, $0 \le i,j,k \le N_0$, there exists a unique solution $\rho^{n+1}$ for the numerical scheme~\eqref{semi_imp}, with discrete periodic boundary condition, with $\rho_{i,j,k}^{n+1} > 0, 0 \le i,j,k \le N_0$.   
\end{prop} 

It is observed that, although the truncation error for~\eqref{semi_imp} is only $O (\dt)$ in the temporal discretization, a one-step computation would lead to an $O (\dt^2)$ approximation 
 to the PDE solution of $\rho_t = \mathcal{B} \rho$ at time step $t^{n+1}$, as long as $\rho^n$ retains a second order temporal accuracy.  
Within the rough guess $\hat{\rho}^{n+1}$, we define $\hat{\rho}^{n+1/2} = \frac{1}{2} (\rho^n + \hat{\rho}^{n+1})$, which is an $O(\dt^2)$ approximation to $\rho$ at the time instant $t^{n+1/2}$. Thus, a second-order accurate scheme can be constructed through Crank-Nicolson type discretization, along with the discrete variational derivative method \cite{du1991numerical, furihata2010discrete}: 
\begin{equation}\label{scheme-Field-2nd-NLB-1}
  \begin{cases}
&\frac{\rho^{n+1} - \rho^n}{\dt} = \nabla_h (\mathcal{M}^{n+1/2}_h \nabla_h \mu^{n+1/2}) , \\
& \mu^{n+1/2} = \frac{F(\rho^{n+1}) - F(\rho^{n}) }{\rho^{n+1} - \rho^n} + \dt (\ln \rho^{n+1} - \ln \rho^{n}) , \\
& \mathcal{M}^{n+1/2}_h = \mathcal{A}_h ( D(\hat{\rho}^{n+1/2}) \hat{\rho}^{n+1/2}), \\
  \end{cases}
\end{equation}
where $F(\rho) = \rho \ln \rho + C \rho$ is the free energy density. Similar to the derivation of~\eqref{2nd_DVD}, the artificial regularization term $\dt (\ln \rho^{n+1} - \ln \rho^{n})$, which does not affect the overall accuracy, is needed in the theoretical justification of the positivity-preserving property; see the following theorem. 
\begin{thm}\label{Field-2nd-NLB-positivity} 
  Given $\rho^n$, with $\rho_{i,j,k}^n > 0$, $\forall 0 \le i, j, k \le N_0-1$, there exists a unique solution $\rho^{n+1}$ for the numerical scheme (\ref{scheme-Field-2nd-NLB-1}), with the discrete periodic boundary condition satisfying $\rho^{n+1}_{i,j,k} > 0$.
\end{thm} 

To simplify the notation, we introduce an average operator: $$\overline{f} = \frac{h^3}{| \Omega|} \sum_{i,j,k=0}^{N-1} f_{i,j,k},$$  and define a hyperplane in $R^{N_0^3}$, with dimension $(N_0^3 -1)$: 
\begin{eqnarray} 
  H = \left\{ \rho_{i,j,k} = \beta_0 + \psi_{i,j,k} : {\textstyle\sum_{i,j,k=0}^{N_0-1}} \psi_{i,j,k}  = 0  \right\} .  \label{Field-def-H}
\end{eqnarray} 
Meanwhile, we recall a preliminary estimate, which has been proved in a recent work ~\cite{chen19b}.  Let $\mathcal{C}_{\Omega}$ be the space of grid function on $\Omega$.
For any
\begin{equation}
  \varphi  \in\mathring{\mathcal C}_{\Omega} = \{ \nu \in \mathcal{C}_{\Omega} | {\bar \nu} = 0 \},
\end{equation}
there exists a unique $\xi \in\mathring{\mathcal C}_{\Omega}$ that solves
	\begin{equation}
\mathcal{L}_{\breve{\mathcal{M}} } (\xi)
 = \varphi, \quad \text{where} \quad \mathcal{L}_{\breve{\mathcal{M}} } (\xi):= - \nabla_h \cdot ( \check{\mathcal{M}} \nabla_h \xi).
	\label{PNP-mobility-0} 
	\end{equation}
In turn, the following discrete norm can be defined: 
	\begin{equation} 
\| \varphi  \|_{\mathcal{L}_{\breve{\mathcal{M}} }^{-1} } = \sqrt{ \langle \varphi ,  \mathcal{L}_{\check{\mathcal{M}} }^{-1} (\varphi) \rangle },  
	\end{equation} 
  which is a discrete weighted $H^{-1}$-norm associated with a non-constant mobility. 

    \begin{lem} \cite{chen19b} 
    \label{Field-mobility-positivity-Lem-0}  
  Suppose that $\varphi_1$, $\varphi_2 \in \mathcal{C}_{\rm per}$, with $\langle \varphi_1 - \varphi_2 , 1 \rangle = 0$, i.e., $\varphi_1 - \varphi_2\in \mathring{\mathcal{C}}_{\rm per}$, and assume that $\| \varphi_1 \|_\infty , \| \varphi_2\|_\infty \le M_h$, and $\mathcal{M} \ge \mathcal{M}_0$ at a point-wise level. Then we have the following inequality:
  \begin{equation}  
   \| \mathcal{L}_{\check{\mathcal{M}} }^{-1} (\varphi_1 - \varphi_2) \|_{\infty} \leq C_2 := \tilde{C}_2 \mathcal{M}_0^{-1} h^{-1/2}.
  \end{equation}
  where $\tilde{C}_2>0$ depends only upon $\mathcal{M}_h$ and $\Omega$.
    \end{lem}

Now we proceed into the proof of Theorem~\ref{Field-2nd-NLB-positivity}.  
\begin{proof} 
  The mass conservative property of the numerical solution~\eqref{scheme-Field-2nd-NLB-1}
  is obvious:  
  \begin{equation}
    \overline{\rho^{n+1}} = \overline{\rho^n} := \beta_0.
  \end{equation}
  A direct calculation implies that, if $\rho^{n+1}$ with $\rho^{n+1}_{i,j,k} > 0$ is the numerical solution of~\eqref{scheme-Field-2nd-NLB-1}, $\rho^{n+1}$
  is a minimization of the following discrete energy functional: 
 \begin{equation}
J_n (\rho) = \frac{1}{2 \dt} \| \rho - \rho^n \|_{\mathcal{L}_{\check{\mathcal{M}}^n}} +  \langle G_{\rho^n}^0 (\rho)  + \dt ( \rho \ln \rho  + C_n \rho),  {\bf 1} \rangle,
 \end{equation}
over the admissible set 
\begin{eqnarray} 
  V_h^H &: = \left\{  \rho = \beta_0 + \psi ~|~ \bar{\psi} =0,    0 < \rho_{i,j,k} < M_h  , \forall ( i, j, k )  \right\}.
  \label{domain-1} 
\end{eqnarray} 
Here $C_n  = C - 1 - \dt (1 + \ln \rho^n)$, $G_{\rho^n}^0 (\rho)$ is defined in (\ref{def_G}), and $M_h = \frac{\beta_0}{h^3}$.

To this end, we consider the following closed domain: 
\begin{eqnarray} 
  V_{h,\delta}^H =  \left\{ \psi : \overline{\psi} = 0, \, \delta \le \rho_{i,j,k}  \le M_h   \right\}   
    \subset V_h^H .  \label{Field-2nd-NLB-positive-2} 
\end{eqnarray}   
Since $V_{h,\delta}^H$ is a bounded, compact set in the hyperplane $H$, there exists a (may not unique) minimizer of $J_h (\psi)$ over $V_{h, \delta}^H$. The key point of the positivity analysis is that, such a minimizer could not occur on the boundary points (in $H$) if $\delta$ is small enough. 

For a given $\rho^n$ with $\rho_{i,j,k}^n > 0$, we can assume that $\rho^n$ satisfies the following bounds 
\begin{equation}  
  \epsilon_0 \le \rho^n_{i,j,k} \le  M_h - \epsilon_0, \quad \forall 0 \le i, j, k  \le N_0 -1 . \label{bound-un} 
\end{equation} 
Assume a minimizer of $J_n (\rho)$ occurs at a boundary point of $V_{h,\delta}^H$. Without loss of generality, we set the minimization point as $\rho^*_{i,j,k}$, with $\rho^*_{i_0,j_0,k_0}=\delta$. In addition, we denote the grid point that $\rho^*$ reaches the maximum value as $(i_1, j_1, k_1)$. 
It is obvious that $\rho^*_{i_1, j_1, k_1} \ge \beta_{0}$, because of the fact that $\overline{\rho^*} = \beta_{0}$. 

To obtain a contradiction, we compute the direction derivative of $J_n(\rho)$ along the direction
\begin{equation}
\delta \psi = \delta_{i,i_0} \delta_{j, j_0} \delta_{k, k0} - \delta_{i, i_1} \delta_{j, j_1} \delta_{k,k_1} \in \mathring{\mathcal{C}}_{\rm per}, \, \,  
\mbox{$\delta_{k, l}$ is the Kronecker delta function} , 
\end{equation}
and the following identity is valid: 
\begin{equation*}
  \begin{aligned}
 & \frac{1}{h^3} \frac{J_n (\rho^* + s \delta\psi) - J_n (\rho^*)}{s} \Big|_{s = 0} = \frac{1}{\Delta t} \left( \mathcal{L}_{\check{\mathcal{M}}^n}(\rho^* - \rho^n) )_{i_0, j_0, k_0} - \mathcal{L}_{\check{\mathcal{M}}^n}(\rho^* - \rho^n) )_{i_1, j_1, k_1}  \right) \\
 & + (G_{\rho^n}^1 (\rho^*))_{i_0, j_0, k_0} - (G_{\rho^n}^1 (\rho^*))_{i_1, j_1, k_1}  +  \dt ( \ln \rho^*_{i_0, j_0, k_0} - \ln \rho^*_{i_1, j_1, k_1} ) + \langle C_n, \delta \psi \rangle . 
  \end{aligned}
\end{equation*}
In addition, by the fact that $\rho^*_{i_0, j_0, k_0}  = \delta$ and $\rho^*_{i_1, j_1, k_1} \geq \beta_0$, we get 
\begin{equation}
  \ln \rho^*_{i_0, j_0, k_0} - \ln \rho^*_{i_1, j_1, k_1} \leq \ln \delta - \ln \beta_0.
\end{equation}
In the meantime, the following inequality could be derived, based on Lemma~\ref{Field-mobility-positivity-Lem-0}:
$$
  \frac{1}{\Delta t} \left| \left( \mathcal{L}_{\check{\mathcal{M}}^n}(\rho^* - \rho^n) )_{i_0, j_0, k_0} - \mathcal{L}_{\check{\mathcal{M}}^n}(\rho^* - \rho^n) )_{i_1, j_1, k_1} \right) \right| 
   \leq 2 \tilde{C}_2 \mathcal{M}_0^{-1} h^{-1/2} \dt^{-1} . 
$$  
Since $G_{a}^1(x)$ is an increasing function in term of $x > 0$ for any fixed $a > 0$, and $\rho^n$ satisfies the bound (\ref{bound-un}), it is straightforward to obtain  
\begin{equation*}
  \begin{aligned}
  & (G_{\rho^n}^1 (\rho^*))_{i_0, j_0, k_0} - (G_{\rho^n}^1 (\rho^*))_{i_1, j_1, k_1}  + \langle C_n, \delta \psi \rangle \\ 
  & \leq \ln M_h + 1  - G_{\epsilon_0}^1 (\beta_0) + \dt (\ln M_h - \ln \epsilon_0) . 
  \end{aligned}
\end{equation*}
As a consequence, a combination of the above estimates leads to 
\begin{equation}
  \frac{1}{h^3} \frac{J_n (\rho^* + s \delta\psi) - J_n (\rho^*)}{s} \Big|_{s = 0} \leq D_0 + \dt (\ln \delta - \ln \beta^0),
\end{equation}
where $D_0= 2 \tilde{C}_2 \mathcal{M}_0^{-1} h^{-1/2} \dt^{-1}  + \ln M_h + 1  - G_{\epsilon_0}^1 (\beta_0) + \dt (\ln M_h - \ln \epsilon_0) $, a constant for fixed $\dt$ and $h$. Hence, we can choose $\delta$ to sufficiently small such that 
\begin{equation}
  \frac{1}{h^3} \frac{J_n (\rho^* + s \delta\psi) - J_n (\rho^*)}{s} \Big|_{s = 0} < 0 . 
\end{equation}
This inequality contradicts with the assumption that $\rho^*$ is a minimizer of $J_n (\rho)$. Therefore, a minimizer of $J_n (\rho)$ cannot occur on the boundary of $V_{h,\delta}^H$ if  $\delta$ is small enough. 
In other words, the minimizer of $J_n (\rho)$ over $V_{h}^H$ could only possibly occur at its interior point, which gives a solution of the numerical scheme \eqref{scheme-Field-2nd-NLB-1}. The uniqueness of this numerical solution comes from a direct application of the strict convexity of $J_n (\rho)$. The proof of Theorem~\ref{Field-2nd-NLB-positivity} is complete. 
\end{proof}

With the positivity-preserving property and the unique solvability established, we can further prove the following unconditional energy stability. 
\begin{thm}
	\label{Field-2nd-NLB-energy stability} 
For the numerical solution~\eqref{scheme-Field-2nd-NLB-1}, we have 
	\begin{eqnarray}  
\mathcal{F}_h (\rho^{n+1})	
   \le \mathcal{F}_h (\rho^n) ,   \quad \mbox{with} \, \, \, 
   \mathcal{F}_h (\rho^n) = \langle \rho^n \ln \rho^n + C \rho^n, {\bf 1} \rangle .    
   \label{Field-2nd-NLB-energy-0} 
	\end{eqnarray} 
	\end{thm}


\begin{proof} 
Taking a discrete inner products with~\eqref{scheme-Field-2nd-NLB-1} by $\mu^{n+1/2}$ yields 
\begin{eqnarray} 
\begin{aligned} 
\frac{1}{\Delta t} \langle  \rho^{n+1} - \rho^n, \mu^{n+1/2}  \rangle = -   \langle \mathcal{M}^{n+1/2}_h \nabla_h \mu^{n+1/2}, \nabla_h \mu^{n+1/2} \rangle \leq 0
\end{aligned} 
\label{Field-2nd-NLB-energy-1} 
\end{eqnarray}
Notice that
\begin{equation}
  \begin{aligned}
  \langle  \rho^{n+1} - \rho^n, \mu^{n+1/2}  \rangle  & = \mathcal{F}_h (\rho^{n+1}) -  \mathcal{F}_h (\rho^{n}) +  \dt \langle \rho^{n+1} - \rho^n, \ln \rho^{n+1} - \ln \rho^n \rangle \\
  & \geq  \mathcal{F}_h (\rho^{n+1}) -  \mathcal{F}_h (\rho^{n}) , 
  \end{aligned}
\end{equation}
due to monotonic property of the logarithmic function. Then we arrive at 
\begin{equation}
  \mathcal{F}_h (\rho^{n+1}) -  \mathcal{F}_h (\rho^{n}) \leq  -   \langle \mathcal{M}^{n+1/2}_h \nabla_h \mu^{n+1/2}, \nabla_h \mu^{n+1/2} \rangle \leq 0.
\end{equation}
\end{proof}

\begin{remark}
  It is worth emphasizing that, the discretization presented in~\eqref{scheme-Field-2nd-NLB-1} is based on the $H^{-1}$-gradient flow structure of the diffusion equations. One can also construct a variational structure preserving scheme for diffusion equations by using the Lagrangian methods \cite{carrillo2018lagrangian, junge2017fully, liu2019lagrangian}, which treat diffusion equations as an $L^2$-gradient flow in the space of diffeomorphism,  or the numerical methods for Wasserstein gradient flows in the space of probability measure \cite{benamou2016augmented}.
  \end{remark}

 \subsection{The second order accurate operator splitting scheme} 

The second-order operator splitting scheme could be formulated as follows, based on the previous analyses. 

 Given ${\bm c}^n$ with ${\bm c}_{i,j,k}^n \in \mathbb{R}^N_{+}$, we update ${\bm c}^{n+1}$ via the following three stages.  

 \noindent 
 {\bf Stage 1.} \, Setting ${\bm c}_0 = {\bm c}^n$ and solving the reaction trajectory equation, subject to the initial condition $R^n = 0$, using scheme (\ref{2nd_DVD}) with a temporal step-size $\dt / 2$. An intermediate numerical profile is updated as 
 \begin{equation}
 {\bm c}^{n+1, (1)} = {\bm c}^{n} + {\bm \sigma} R^{n+1, (1)}.
 \end{equation}

 \noindent 
 {\bf Stage 2.} \, Starting with the intermediate variable ${\bm c}^{n+1, (1)}$, we solve the diffusion equation $\pp_t {\bm c} = \mathcal{B} {\bm c}$ by applying either scheme (\ref{ETD_Heat}) (for constant diffusion coefficient) or scheme (\ref{scheme-Field-2nd-NLB-1}) (for nonlinear diffusion coefficient), with a temporal step-size $\dt$, to obtain ${\bm c}^{n+1, (2)}$.
 
 \noindent
 {\bf Stage 3.} \, We set ${\bm c}_0 = {\bm c}^{n+1, (2)}$ and repeat the numerical algorithm at stage 1, i.e., solving the reaction trajectory equation, subject to the initial condition $R^n = 0$, by scheme (\ref{2nd_DVD}) with the temporal step-size $\dt / 2$ to obtain  $R^{n+1, (2)}$. The numerical solution at $t^{n+1}$ is updated as 
 \begin{equation}
   {\bm c}^{n+1} = {\bm c}^{n+1, (2)} + {\bm \sigma} R^{n+1, (2)}.
   \end{equation}
The following theoretical result for the second-order operator splitting scheme can be established, based on Theorem 3.1 - 3.6. 

\begin{thm}  \label{Field-operator splitting} 
Given ${\bm c}^n$ with ${\bm c}_{i, j, k}^n \in \mathbb{R}^N_{+}, \forall 0 \le i, j, k \le N_0-1$ and a discrete period boundary condition,, there exists a unique solution ${\bm c}^{n+1}$ with ${\bm c}_{i, j, k}^{n+1} \in \mathbb{R}^N_{+}, \forall 0 \le i, j, k \le N_0-1$,  for the second order accurate operator splitting numerical scheme. In addition, we have the energy dissipation estimate: $$\mathcal{F}_h ({\bm c}^{n+1}) \le \mathcal{F}_h ({\bm c}^n) ,$$ so that $\mathcal{F}_h ({\bm c}^n) \le \mathcal{F}_h({\bm c}^0)$, a constant independent of $h$.  
\end{thm}

\section{The numerical results}

\subsection{Reaction kinetics}
In this subsection, we test the accuracy order for the algorithm~\eqref{2nd_DVD}, by considering a simple reaction kinetics (with $\alpha >0$): 
\begin{equation}\label{test_1}
\begin{cases}
  & \dfrac{\dd c_1}{\dd t} = c_2 - \alpha c_1 , \\
  & \dfrac{\dd c_2}{\dd t} = \alpha c_1 - c_2  .   
\end{cases}
\end{equation}
In fact, this equation corresponds to a simple reversible chemical reaction $\ce{X_1 <=>[a][1] X_2}$. For any given initial value $c_i(0) =c_i^0$, the exact solution turns out to be 
\begin{equation}
c_1(t) = \left(1 + \left( \frac{c_1^0}{c_1^{\infty}} -1 \right) \exp(-(a+1)) t \right) c_1^{\infty}, \quad c_2(t) = c_1^0 + c_2^0 - c_1(t),
\end{equation}
with $c_1^{\infty} = (c_1^0 + c_2^0) / (\alpha + 1)$ being the equilibrium concentration of $X_1$. Following the earlier analysis, we introduce $R$ as the reaction trajectory, 
so that the energy-dissipation law becomes 
\begin{equation}
\frac{\dd}{\dd t} \left( \sum_{i=1}^2 c_i (\ln c_i - 1) + c_1 \ln a + c_2 \ln (1)  \right) = - \dot{R} \ln \left(  \frac{\dot{R}}{c_2} + 1 \right).
\end{equation}

To test the numerical accuracy order, we display the errors between the numerical solution and exact solution at $T = 1$ in Table~\ref{table1}, with a sequence of step sizes $\dt$.  An almost perfect second order temporal accuracy is observed. 

\begin{table}[!h]
\begin{center}
  \begin{tabular}{c|c|c} 
    \hline
   $\dt$ & Error & Order  \\
   \hline
   1/20  &  2.0882e-3  &\\ 
   1/40  &  5.3413e-4  &  1.9670 \\  
   1/80  &   1.3577e-4  & 1.9760  \\
   1/160 &  3.4279e-5  &  1.9858  \\
   1/320 &  8.6159e-6   &  1.9923  \\
   1/640 &  2.1600e-06 &  1.9960 \\
   \hline
   \hline
 \end{tabular}
\end{center}
\caption{Error table for the linear ODE system (\ref{test_1})}\label{table1}
\end{table}

   %

\subsection{Reaction-diffusion systems}
In this subsection, we consider the reaction-diffusion system
\begin{equation}\label{example3}
\begin{cases}
  & \pp_t u = D_u \Delta u^{\alpha}  - k_1^+ u v^2 + k_1^- v^3 \\
  &  \pp_t v = D_v \Delta v + k_1^+ u v^2 - k_1^- v^3,
\end{cases}
\end{equation}
where $\alpha \geq 1$ is a constant,  $D_u > 0$ and $D_v > 0$ are diffusion coefficients.
The reaction part of (\ref{example3}) describes the chemical reaction
$$ 
  \ce{U + 2 V <=>[k_1^+][k_1^-] 3V}.
$$ 
with the law of mass action. The the whole system satisfies the energy-dissipation law 
\begin{equation*}
  \begin{aligned}
  & \frac{\dd}{\dd t} \int_\Omega u (\ln u - 1 + U_u) + v (\ln v - 1 + U_v)  \dd \x \\
  &  = -  \int_\Omega \dot{R} \ln \left( \frac{\dot{R}}{k_1^{-} v^{3}} + 1\right) + \alpha u^{\alpha} D_u |\nabla \mu_u|^2 + D_v |\nabla \mu_v|^2  \dd \x.
  \end{aligned}
\end{equation*}
The internal energies can be taken as $U_u = \ln k_1^+$ and $U_v = \ln k_1^-$ so that $\dot{R} = k_1^+ u v^2 - k_1^- v^3$.

\begin{figure}[!t]
  \begin{overpic}[width = \linewidth]{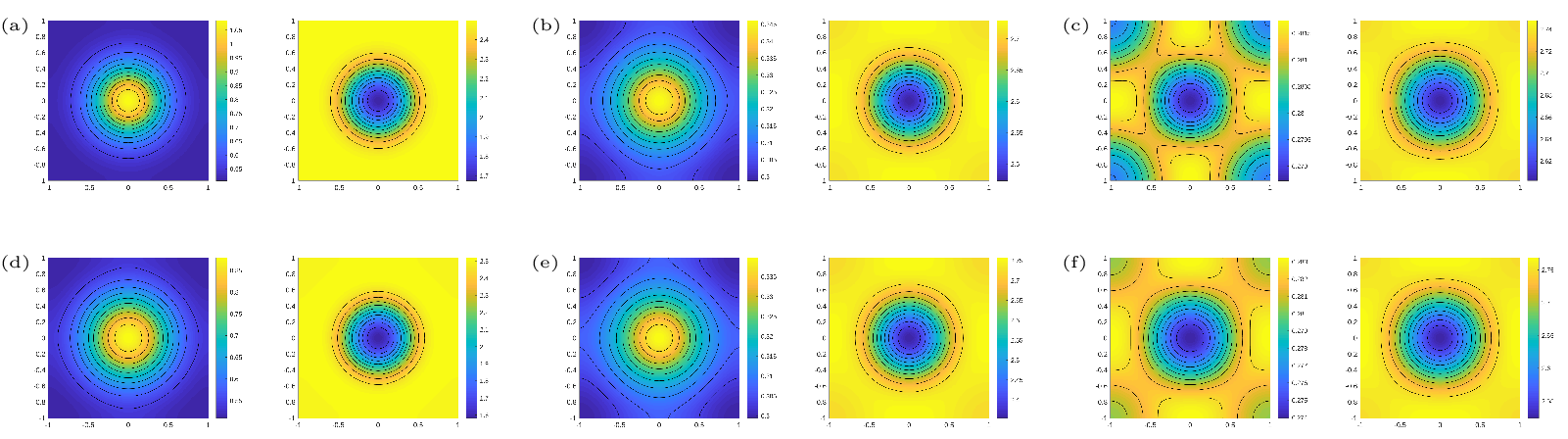}
  \end{overpic}
\caption{Numerical solutions for the reaction-diffusion system (\ref{example3}) with $\alpha = 1$ (a - c) and $\alpha =2$ (d - f) at $t = 0.2$ (a and d), $t = 0.5$ (b and e) and $t = 0.7$ (c and f).}\label{Fig1}
\end{figure}

For $\alpha = 1$, we apply the ETD scheme (\ref{ETD_Heat}) to solve the diffusion parts for both $u$ and $v$. Otherwise we use scheme (\ref{scheme-Field-2nd-NLB-1}) for $u$ and use the ETD scheme for $v$.  
The computational domain is taken as $\Omega = (-1, 1)^2$, and a periodic boundary condition is imposed for both $u$ and $v$. The initial value is set as 
\begin{equation*}
  \begin{aligned}
& u = (- \tanh( (\sqrt{x^2 + y^2} - 0.4)/0.1) + 1)/2 + 1; \\
& v = (\tanh( (\sqrt{x^2 + y^2} - 0.4)/0.1) + 1)/2 + 1. \\
  \end{aligned} 
\end{equation*}
Other parameters are taken as: $D_u = 0.2$, $D_v = 0.1$, $k_1^+ = 1$ and $k_1^- = 0.1$. 

Fig. \ref{Fig1} shows the numerical solutions at $t = 0.2, 0.5$ and $0.7$ for $\alpha = 1$ and $\alpha =2$ respectively, which are obtained by taking $h = \Delta t= 1/20$. The discrete free energy evolutions corresponding to these two numerical solutions are displayed in Fig. \ref{Fig2}, 
which clearly demonstrate the energy stability of the operator splitting scheme in both linear and nonlinear diffusion cases.
\begin{figure}[!h]
  \centering
  \includegraphics[width = 0.6 \linewidth]{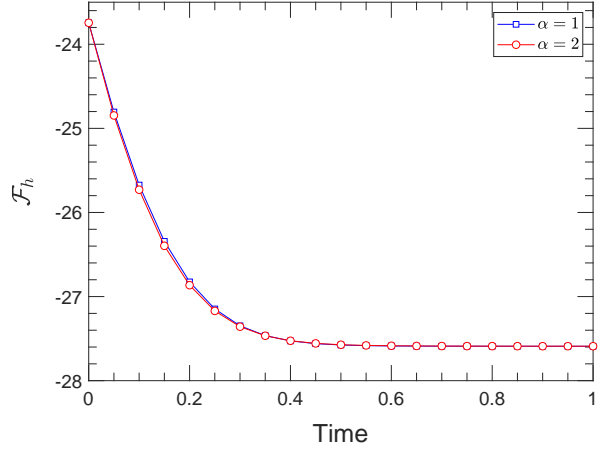}
  \caption{The discrete free energy evolutions corresponding to numerical solutions for the reaction-diffusion system (\ref{example3}) with $\alpha = 1$ and $\alpha =2$ ($h = \Delta t= 1/20$).}\label{Fig2}
  \end{figure}

Next we test for numerical accuracy of the operator splitting scheme. As analytical forms of the exact solutions are not available, we perform a Cauchy convergence test for numerical simulations for $\alpha = 1$ and $\alpha = 2$, respectively, at $T = 0.2$, before the systems reach their constant equilibria. We compute the $\ell^{\infty}$ differences between numerical solutions with consecutive spatial resolutions, $h_{j-1}$, $h_j$ and $h_{j+1}$, with $\Delta t_j = h_j$.
Since we expect the numerical scheme preserves a second order spatial accuracy, the following quantity could be computed 
$$
   \frac{ \ln \Big(  \frac{1}{A^*} \cdot 
   \frac{\| u_{h_{j-1}} - u_{h_j} \|_\infty }{ \| u_{h_j} - u_{h_{j+1}} \|_\infty} \Big) } 
   {\ln  \frac{h_{j-1}}{h_j} } ,  \quad A^* =  \frac{ 1 - \frac{h_j^2}{h_{j-1}^2} }{1 - \frac{h_{j+1}^2}{h_j^2} } ,  
   \quad \mbox{for} \, \, \, h_{j-1} > h_j > h_{j+1} , 
$$
to check the convergence order \cite{liu2020positivity}. As demonstrated in Tables~\ref{t2:convergence} and \ref{t3:convergence}, 
an almost perfect second order accuracy has been achieved for both the linear and nonlinear diffusion cases.

\begin{table}[ht]
  \begin{center}
  \begin{tabular}{c |c |c |c |c }
  \hline  \hline
   --- & $\psi = u$ & Order & $\psi = v$& Order  \\
   \hline
  $\| \psi_{h_1} - \psi_{h_2} \|_{\infty}$  & 4.1625e-3 & -       & 3.6818e-3  & -    \\
  $\| \psi_{h_2} - \psi_{h_3} \|_{\infty}$  & 1.5357e-3 & 1.8700  & 1.3581e-3  & 1.8705  \\ 
  $\| \psi_{h_3} - \psi_{h_4} \|_{\infty}$  & 7.3080e-4 & 1.9036  & 6.4788e-4  & 1.8950    \\
  $\| \psi_{h_3} - \psi_{h_4} \|_{\infty}$  & 4.0386e-4 & 1.9230  & 3.5830e-4 & 1.9197     \\
   \hline  \hline
  \end{tabular}
  \caption{The $\ell^\infty$ differences and convergence order for the numerical solutions of $u$, and $v$ for (\ref{example3}) with $\alpha =1$. Various mesh resolutions are used: $h_1=\frac{1}{20}$, $h_2=\frac{1}{30}$, $h_3=\frac{1}{40}$, $h_4=\frac{1}{50}$, $h_5=\frac{1}{60}$, and the time step size is taken as $\Delta t_j = h_j$.}
  \label{t2:convergence}
  \end{center}
  \end{table}

\begin{table}[ht]
  \begin{center}
  \begin{tabular}{c |c |c |c |c }
  \hline  \hline
   --- & $\psi = u$ & Order & $\psi = v$& Order  \\
   \hline
  $\| \psi_{h_1} - \psi_{h_2} \|_{\infty}$  & 4.4205e-3 & -       & 2.6961e-3  & -    \\
  $\| \psi_{h_2} - \psi_{h_3} \|_{\infty}$  &   1.4508e-3 & 2.1586 & 9.4864e-4  & 1.9870   \\ 
  $\| \psi_{h_3} - \psi_{h_4} \|_{\infty}$  & 6.1387e-4 & 2.3120  & 4.3720e-4  & 2.0150   \\
  $\| \psi_{h_3} - \psi_{h_4} \|_{\infty}$  & 3.1575e-4 & 2.2446  & 2.4420e-4 & 1.8752   \\
   \hline  \hline
  \end{tabular}
  \caption{The $\ell^\infty$ differences and convergence order for the numerical solutions of $u$, and $v$ for (\ref{example3}) with $\alpha =2$ at $T=0.2$. Various mesh resolutions are used: $h_1=\frac{1}{20}$, $h_2=\frac{1}{30}$, $h_3=\frac{1}{40}$, $h_4=\frac{1}{50}$, $h_5=\frac{1}{60}$, and the time step size is taken as $\Delta t_j = h_j$.}
  \label{t3:convergence}
  \end{center}
  \end{table}

\section{Concluding remarks}
A second-order accurate, operator splitting numerical scheme is developed for reaction-diffusion equations with the  detailed balance condition based on their variational structures. The key idea is to design an operator splitting scheme such that each stage dissipates the same free energy, according to the variational structure associated with the original system. In the reaction part, the reaction trajectory equation is solved by using the numerical techniques from $L^2-$gradient flows, based on a modified Crank-Nicolson approach. In the diffusion part, an ETD algorithm gives an exact time integration for a linear diffusion process, while a semi-implicit algorithm is applied for a nonlinear diffusion. A combination of the numerical algorithms at both stages by the Strang splitting approach leads to the proposed operator splitting scheme. Moreover, the unique solvability, positivity-preserving property, as well as an unconditionally energy stability can be proved for each stage; as a result, the combined splitting scheme also satisfies these theoretical properties. Similar ideas can be applied to other dissipative systems with multiple dissipation mechanisms. A few numerical results have also been presented to demonstrate the numerical performance. 

\section*{Acknowledgement} 
This work is partially supported by the National Science Foundation (USA) grants NSF DMS-1759536, NSF DMS-1950868 (C. Liu, Y. Wang), and NSF DMS-2012669 (C. Wang).  Y. Wang would also like to thank Department of Applied Mathematics at Illinois Institute of Technology for their generous support and for a stimulating environment.

\bibliographystyle{siam}
\bibliography{KCR}

\end{document}